\documentclass[reqno]{amsproc}

\usepackage{hyperref}

\usepackage[latin1]{inputenc}
\usepackage[T1]{fontenc}
\usepackage[francais, english]{babel}
\usepackage{amssymb,amsmath}
\usepackage{amsthm}
\usepackage{mathrsfs}

\title[M. El Berrag  and  A. Tajmouati   ]
{   Ces\`{a}ro-Hypercyclic  and  Weyl type theorem  }\vspace{3cm}

\author[ M. El Berrag  and  A. Tajmouati ]
{\hspace{0.5cm} Abdelaziz Tajmouati  \hspace{0.5cm} Mohammed El Berrag }


\date{}
\newtheorem{theorem}{\textbf{Theorem}}[section]
\newtheorem{definition}{\textbf{Definition}}[section]

\newtheorem{lemma}{\textbf{Lemma}}[section]
\newtheorem{proposition}{\textbf{Proposition}}[section]
\newtheorem{corollary}{\textbf{Corollairy}}[section]
\newtheorem{example}{\textbf{Example}}

\begin{document}
$      $
\maketitle
\centerline{Sidi Mohamed Ben Abdellah University}
\centerline{ Faculty of Sciences Dhar El Mahraz} \centerline {Fez,
Morocco}
\hspace{1cm}Email: abdelaziztajmouati@yahoo.fr\,\,\,\, med.elberrag@gmail.com\hspace{1cm}

\begin{abstract}
In this paper we  study the relations between Ces$\grave{a}$ro-hypercyclic
operators  and the operators for which Weyl type theorem holds.
\end{abstract}

\section{Introduction}

Throughout this note let $B(\mathcal{H})$ denote the algebra of bounded linear operators acting on a complex, separable,
infinite dimensional Hilbert space $\mathcal{H}.$ If
$T\in B(\mathcal{H}),$  write $N(T)$  and $R(T)$  for the null space and the
range of $T;$  $\sigma(T)$  for the spectrum of $T;$
$\pi_{00}(T) = \pi_{0}(T)\cap$ iso$\sigma(T),$
where  $\pi_{0}(T)= \{\lambda\in \mathbb{C}: 0<\dim N(T-\lambda I)<\infty \}$ are the eigenvalues of finite multiplicity. Let $p_{00}(T)$ denote  the set
of Riesz points of $T$ (i.e., the set of $\lambda\in \mathbb{C}$ such that $T- \lambda I$  is Fredholm of finite ascent and
descent \cite{a}). An operator $T\in B(\mathcal{H})$ is called upper semi-Fredholm if it has closed range with finite dimensional null space and if $R(T)$  has finite co-dimension, $T\in B(\mathcal{H})$ is
called a lower semi-Fredholm operator. We call $T\in B(\mathcal{H})$ Fredholm if it has closed range with
finite dimensional null space and
its range is of finite co-dimension. The index of a Fredholm operator
$T\in B(\mathcal{H})$ if given by
\begin{center}
 ind$(T)= \dim N(T) - \dim R(T)^{\bot}(= \dim N(T) - \dim N(T^{\ast})).$
\end{center}
An operator $T\in B(\mathcal{H})$ is called Weyl if it is Fredholm of index zero. And $T\in B(\mathcal{H})$ is called Browder if it is Fredholm  of finite ascent and descent: equivalently  \cite{h} if  $T$
is Fredholm and $T - \lambda I$ is invertible for sufficiently small $\lambda\neq0$ in $\mathbb{C}.$ The essential spectrum $\sigma_{e}(T),$  the Weyl spectrum $\sigma_{w}(T),$ the Browder spectrum $\sigma_{b}(T),$ the upper semi-Fredholm spectrum and the lower semi-Fredholm spectrum of $T\in B(\mathcal{H})$ are defined by
\begin{quote}
$\sigma_{e}(T) = \{\lambda\in \mathbb{C}: T-\lambda I$  is not Fredholm$\},$  \\
  $\sigma_{w}(T) = \{\lambda\in \mathbb{C}: T-\lambda I$ is not Weyl$\},$ \\
  $\sigma_{b}(T) = \{\lambda\in \mathbb{C}: T-\lambda I$ is not Browder$\},$\\
  $\sigma_{SF_{+}}(T) = \{\lambda\in \mathbb{C}: T-\lambda I$ is not upper semi-Fredholm$\},$  \\
  $\sigma_{SF_{-}}(T) = \{\lambda\in \mathbb{C}: T-\lambda I$ is not lower semi-Fredholm$\}.$
\end{quote}
In keeping with current usage \cite{a, hl}, we say that an operator $T\in B(\mathcal{H})$ satisfies
Browder's theorem (respectively Weyl's theorem) if
$\sigma(T)\backslash \sigma_{w}(T) = p_{00}(T),$
equivalently  $\sigma_{w}(T)= \sigma_{b}(T)$ (respectively $\sigma(T)\backslash
\sigma_{w}(T) = \pi_{00}(T)).$
The following implications hold \cite{hl}: Weyl's theorem for $T$  $\Rightarrow$
Browder's theorem for $T$  $\Rightarrow$  Browder's theorem for $T^{\ast}.$ Let
$\pi_{00}^{a}(T)$ denote the set of $\lambda \in \mathbb{C }$ such that $\lambda$
is an isolated point of $\sigma_{a}(T), \lambda \in$ iso$\sigma_{a}(T),$ and
$0<\dim N(T-\lambda I)<\infty,$  where $\sigma_{a}(T)$  denotes the approximate
point spectrum of the operator $T.$ Then $p_{00}(T)\subseteq \pi_{00}(T)
\subseteq \pi_{00}^{a}(T).$  $T$ is said to satisfy a-Weyl's theorem if
$\sigma_{a}(T)\backslash \sigma_{ea}(T) = \pi_{00}^{a}(T),$ where we write
$\sigma_{ea}(T)$ for the essential approximate point spectrum of $T$
(i.e., $\sigma_{ea}(T)= \bigcap\{\sigma_{a}(T+ K): K\in K(H)\}:$
a-Weyl's theorem for $T$ $\Rightarrow$  Weyl's theorem for $T,$
but the converse is generally false \cite{r1}. It is well known that $\sigma_{ea}(T)$ coincides with $\sigma_{ea}(T) = \{\lambda\in \mathbb{C}: T-\lambda I\not\in SF_{+}^{-}\},$ where  $SF_{+}^{-}(\mathcal{H})=\{T\in B(\mathcal{H}): T$ is upper semi-Fredholm of ind$(T)\leq0\}.$  We say that $T$ satisfies a-Browder's  if $\sigma_{ea}(T)= \sigma_{ab}(T),$ (equivalently, $\sigma_{a}(T)\backslash \sigma_{ea}(T) = p_{00}^{a}(T),$ where  $p_{00}^{a}(T)= \{\lambda\in$ iso$\sigma_{a}(T): \lambda \in p_{00}(T)\}$ \cite{r2} and $\sigma_{ab}(T)$ the Browder essential approximate point spectrum. Evidently, a-Browder's theorem
implies Browder's theorem (but the converse is generally false).\\
We turn to a variant of the essential approximate point spectrum. $T\in B(\mathcal{H})$ is called a generalized upper semi-Fredholm operator if there exists $T$-invariant subspaces $M$ and $N$ such that $\mathcal{H}= M\oplus N$ and $T_{|M}\in SF_{+}^{-}(M), T_{|N}$ is quasinilpotent. Clearly, if $T$ is generalized upper semi-Fredholm, there exists $\epsilon>0$  such that $T-\lambda I \in SF_{+}^{-}(\mathcal{H})$ and
$N(T-\lambda I)\subseteq \bigcap_{n=1}^{\infty}R[(T-\lambda I)^{n}]$  if $0<|\lambda|<\epsilon.$ Clearly, if  $\lambda \in$ iso$\sigma(T), T-\lambda I$ is generalized upper semi-Fredholm. The new spectrum set is defined as follows. Let
\begin{center}
 $\rho_{1}(T)= \{\lambda \in \mathbb{C}:$ there exists  $\epsilon>0$ such that $T-\mu I$  is generalized upper semi-Fredholm if  $0<|\mu - \lambda|<\epsilon \}$
 \end{center}
and let $\sigma_{1}(T)= \mathbb{C}\backslash \rho_{1}(T).$  Then
\begin{quote}
$\sigma_{1}(T)\subseteq \sigma_{ea}(T)\subseteq \sigma_{ab}(T)\subseteq \sigma_{a}(T).$
\end{quote}
$T$ is called approximate isoloid (a-isoloid) (or isoloid) if  $\lambda \in$ iso$\sigma_{a}(T)($iso$\sigma(T))\Rightarrow N(T-\lambda I)\neq\{0\}$ and $T$ is called finite approximate isoloid ($f$-a-isoloid) (or finite isoloid, $f$-isoloid)
operator if the isolated points of approximate point spectrum (of the spectrum) are all eigenvalues of finite multiplicity. Clearly, $f$-a-isoloid
implies a-isoloid and finite isoloid, but the converse is not true.\\
Recall that an operator $T\in B(\mathcal{H})$  has the single-valued extension property at a point $\lambda_{0}\in \mathbb{C},$  SVEP at $\lambda_{0}$ for short, if for every open disc $\mathcal{D}_{\lambda_{0}}$ centered at $\lambda_{0}$ the only analytic function  $f: \mathcal{D}_{\lambda_{0}} \rightarrow H$
satisfying $(T-\lambda I)f(\lambda)= 0$ is the function $f\equiv 0.$  $T$
has SVEP if it has SVEP at every point of $\mathbb{C}$
(= the complex plane). It is known \cite[Lemma 2.18]{d}
that a Banach space operator $T$  with SVEP satisfies
a-Browder's theorem. Our first observation is that for
operators $T\in CH,$  both $T$ and $T^{\ast}$ satisfy a-Browder's theorem.\\
A bounded linear operator $T : \mathcal{H} \rightarrow \mathcal{H}$
is called hypercyclic if there is some vector $x \in \mathcal{H}$  such that $Orb(T,x)
= \{T^{n}x: n\in \mathbb{N} \}$  is dense in $\mathcal{H}$,
where such a vector $x$ is said hypercyclic  for $T.$\\
The first example of hypercyclic operator was given by Rolewicz in \cite{r}. He
proved that if $B$ is a backward shift on the Banach space $l^{p}$, then $\lambda B$ is hypercyclic if and
only if $|\lambda|> 1.$\\
Let $\{e_{n}\}_{n\geq0}$ be the canonical basis
 of $l^{2}(\mathbb{N}).$ If $\{w_{n}\}_{n\in\geq1}$ is a bounded
sequence in $\mathbb{C}\backslash
\{0\},$ then the unilateral backward weighted
shift  $T:\,l^{2}(\mathbb{N})\longrightarrow\, l^{2}(\mathbb{N})$ is defined by $Te_{n}= w_{n}e_{n-1},\,\,\,\, n\geq1, Te_{0}=0,$ and let $\{e_{n}\}_{n\in \mathbb{Z}}$ be the canonical basis  of $l^{2}(\mathbb{Z}).$ If $\{w_{n}\}_{n\in\mathbb{Z}}$ is a bounded
sequence in $\mathbb{C}\backslash
\{0\},$ then the bilateral  weighted
shift  $T:\,l^{2}(\mathbb{Z})\longrightarrow\, l^{2}(\mathbb{Z})$ is defined by
$Te_{n}= w_{n}e_{n-1}.$\\
The definition and the properties of supercyclicity operators were
introduced by Hilden and Wallen \cite{hw}. They proved that all
unilateral backward weighted shifts on a Hilbert space are
supercyclic.\\
A bounded linear
operator $T\in B(\mathcal{H})$ is called
supercyclic if there is some vector $x \in \mathcal{H}$  such that the projective orbit
$\mathbb{C}.Orb(T,x) = \{\lambda T^{n}x: \lambda\in \mathbb{C}, n\in \mathbb{N} \}$ is dense in $X$. Such a vector $x$  is said supercyclic  for $T.$  Refer to
\cite{bm}\cite{kp}\cite{et}\cite{te} for more informations  about
hypercyclicity and supercyclicity.\\
In \cite{s1} and \cite{s}, Salas characterized the bilateral weighted shifts that are hypercyclic
and those that are supercyclic in terms of their weight sequence.
In \cite{f}, N. Feldman  gave a characterization of the invertible bilateral weighted
shifts that are hypercyclic or supercyclic.\\
For the following theorem, see \cite[Theorem 4.1]{f}.
\begin{theorem}\label{A1} Suppose that $T:\,l^{2}(\mathbb{Z})\longrightarrow\, l^{2}(\mathbb{Z})$ is a bilateral weighted shift with
weight sequence $(w_{n})_{n\in \mathbb{Z}}$  and either $w_{n}\geq m>0$  for all $n<0$ or $w_{n}\leq m$  for all $n>0.$ Then:
\begin{enumerate}
  \item  $T$ is hypercyclic if and only if there exists a sequence of integers $n_{k} \rightarrow \infty$ such
that $\lim_{k\rightarrow \infty}\prod_{j=1}^{n_{k}}w_{j}= 0$ and $\lim_{k\rightarrow \infty}\prod_{j=1}^{n_{k}}\frac{1}{w_{-j}}= 0.$
  \item  $T$ is supercyclic if and only if there exists a sequence of integers $n_{k} \rightarrow \infty$ such
that $\lim_{k\rightarrow \infty}(\prod_{j=1}^{n_{k}}w_{j})(\prod_{j=1}^{n_{k}}\frac{1}{w_{-j}})= 0.$

\end{enumerate}
\end{theorem}
Let $\mathcal{M}_{n}(T)$  denote the arithmetic mean of
the powers of $T\in B(\mathcal{H})$, that is $$ \mathcal{M}_{n}(T)= \frac{1+ T+ T^{2}+ ... T^{n-1}}{n}, n\in \mathbb{N}^{\ast}.$$
If the arithmetic
means of the orbit of $x$ are dense in $\mathcal{H}$ then the operator $T$  is said to be Ces$\grave{a}$ro-hypercyclic. In \cite{l}, Fernando Le$\acute{o}$n-Saavedra proved that an operator is Ces$\grave{a}$ro-hypercyclic if and only if there exists a vector $x\in \mathcal{H }$ such that the orbit $\{n^{-1}T^{n}x\}_{n\geq1}$
is dense in $\mathcal{H}$ and characterized the bilateral weighted
shifts that are Ces$\grave{a}$ro-hypercyclic.\\
For the following proposition, see \cite[Proposition 3.4]{l}.
\begin{proposition}\label{A2} Let  $T:\,l^{2}(\mathbb{Z})\longrightarrow\, l^{2}(\mathbb{Z})$ be a bilateral weighted shift with
weight sequence $(w_{n})_{n\in \mathbb{Z}}.$ Then $T$ is Ces$\grave{a}$ro-hypercyclic if and only if there exists an
increasing sequence $n_{k}$ of positive integers such that for any integer $q,$ \begin{center}
  $\lim_{k\rightarrow \infty}\prod_{i=1}^{n_{k}}\frac{w_{i+q}}{n_{k}}= \infty$ and\, $\lim_{k\rightarrow \infty}\prod_{i=0}^{n_{k}-1}\frac{w_{q-i}}{n_{k}}= 0.$
\end{center}
\end{proposition}
Hypercyclic and supercyclic (Hilbert space) operators satisfying a Browder-Weyl
type theorem have recently been considered by
Cao \cite{c}. In \cite{d1} B.P. Duggal gave the necessary and sufficient conditions
for hypercyclic and supercyclic operators to satisfy a-Weyl's theorem.\\
In this paper we will give an example of a hypercyclic and supercyclic
operator which is not Ces$\grave{a}$ro-hypercyclic and vice versa. Furthermore, we study the relations between Ces$\grave{a}$ro-hypercyclic
operators  and the operators for which Weyl type theorem holds.
\section{  Main results  }
Suppose $\{n^{-1}T^{n}: n\geq1 \}$ is a sequence of bounded linear operators
on $\mathcal{H}$
\begin{definition} An operator $T\in B(\mathcal{H})$  is Ces$\grave{a}$ro-hypercyclic if and only if there exists a vector $x\in \mathcal{H }$ such that the orbit $\{n^{-1}T^{n}x\}_{n\geq1}$  is dense in $\mathcal{H}$
\end{definition}
The following example gives an operator which is
Ces$\grave{a}$ro-hypercyclic but not hypercyclic.
\begin{example}\cite{l} Let $T$ the bilateral backward shift with the weight sequence
$$w_{n}=\left\lbrace
\begin{array}{ll}
1 & \mbox{if $n\leq0,$ }\\
2 & \mbox{if $n\geq1.$}
\end{array}
\right.$$
Then $T$ is not hypercyclic, but it is Ces$\grave{a}$ro-hypercyclic.
\end{example}
Now, we will give an example of a hypercyclic and supercyclic
operator which is not Ces$\grave{a}$ro-hypercyclic.
\begin{example} Let $T$ the bilateral backward shift with the weight sequence
$$w_{n}=\left\lbrace
\begin{array}{ll}
2 & \mbox{if $n<0,$ }\\
\frac{1}{2} & \mbox{if $n\geq0.$}
\end{array}
\right.$$
Then $T$ is not Ces$\grave{a}$ro-hypercyclic, but it is hypercyclic and supercyclic.
\end{example}
\begin{proof} By applying Theorem \ref{A1} and taking $n_{k}=n,$ we have
$$\lim_{n\rightarrow \infty}\prod_{j=1}^{n}w_{j}= \lim_{n\rightarrow \infty}\frac{1}{2^{n}}= 0;$$
and $$\lim_{n\rightarrow \infty}\prod_{j=1}^{n}\frac{1}{w_{-j}}= \lim_{n\rightarrow \infty}\frac{1}{2^{n}} = 0.$$
Furthermore, we have $$\lim_{n\rightarrow \infty}(\prod_{j=1}^{n}w_{j})(\prod_{j=1}^{n}\frac{1}{w_{-j}})= \lim_{n\rightarrow \infty}(\frac{1}{2^{n}})(\frac{1}{2^{n}})=0.$$
Therefore by Theorem \ref{A1} the operator $T$  is hypercyclic and supercyclic. However, for all increasing sequence $n_{k}=n$ of
positive integers and taking $q=0$, we have

$$\lim_{n\rightarrow \infty}\prod_{i=1}^{n}\frac{w_{i+q}}{n}= \lim_{n\rightarrow \infty}\frac{1}{n2^{n}}=0,$$
from Proposition \ref{A2}, $T$  is not Ces$\grave{a}$ro-hypercyclic.
\end{proof}
The following example gives us an operator which is Ces$\grave{a}$ro-hypercyclic but not hypercyclic and supercyclic.
\begin{example} Let $T$ the bilateral backward shift with the weight sequence
$$w_{n}=\left\lbrace
\begin{array}{ll}
\frac{1}{2} & \mbox{if $n<0,$ }\\
n+1 & \mbox{if $n\geq0.$}
\end{array}
\right.$$
Then $T$ is Ces$\grave{a}$ro-hypercyclic, but it is not  hypercyclic and supercyclic.
\end{example}
\begin{proof} By applying Proposition \ref{A2} and taking $n_{k}=n$ and $q=0,$ we have
$$\lim_{n\rightarrow \infty}\prod_{i=1}^{n}\frac{w_{i+q}}{n}= \lim_{n\rightarrow \infty}\frac{(n+1)!}{n}=\infty,$$ and
$$\lim_{n\rightarrow \infty}\prod_{i=0}^{n}\frac{w_{q-i}}{n}= \lim_{n\rightarrow \infty}\frac{1}{n2^{n}}=0.$$
Therefore by Proposition \ref{A2} the operator $T$  is Ces$\grave{a}$ro-hypercyclic. On the other hand, we have
$$\lim_{n\rightarrow \infty}\prod_{j=1}^{n}w_{j}= \lim_{n\rightarrow \infty}((n+1)!)= \infty;$$ and $$\lim_{n\rightarrow \infty}(\prod_{j=1}^{n}w_{j})(\prod_{j=1}^{n}\frac{1}{w_{-j}})= \lim_{n\rightarrow \infty}((n+1)!)(2^{n})=\infty.$$
Therefore by Theorem \ref{A1} the operator $T$  is not  hypercyclic and supercyclic.

\end{proof}
We denote by $CH(\mathcal{H})$  the set
of all ces$\grave{a}$ro-hypercyclic  operator  in $B(\mathcal{H})$ and $\overline{CH(\mathcal{H})}$  the norm-closure of the class $CH(\mathcal{H})$. The following lemma \cite[Theorem 5.1]{l} give  the essential facts for hypercyclic operators and supercyclic operators that we will need to prove the main theorem.


\begin{lemma}\label{B2} $\overline{CH(\mathcal{H})}$ is the class of all those operators $T\in B(\mathcal{H})$ satisfying the conditions:
\begin{enumerate}
  \item $\sigma_{w}(T)\cup  \partial D$ is connected;
  \item $\sigma(T)\backslash\sigma_{b}(T)=\emptyset;$
  \item ind$(T- \lambda I)\geq0$ for every $\lambda\in \rho_{SF}(T),$ where $\rho_{SF}(T)= \{\lambda \in \mathbb{C}: T- \lambda I$ is semi-Fredholm $\}.$
\end{enumerate}

\end{lemma}

Let $H(T)$ be the class of complex-valued functions which are analytic in a neighborhood of $\sigma(T)$  and are not constant on any neighborhood of any component of $\sigma(T)$. Our results are:\\
\begin{theorem} If $T\in B(\mathcal{H})$  is $f$-isoloid and the Weyl's theorem holds for $T$ (or $T$ is $f$-a-isoloid
and the a-Weyl's theorem holds for $T$ ), then  $T\in \overline{CH(\mathcal{H})} \Leftrightarrow \sigma(T)= \sigma_{1}(T)$ and $\sigma(T)\cup \partial D$ is connected

\end{theorem}
\begin{proof}  Suppose $T\in \overline{CH(\mathcal{H})}.$ Let $\lambda_{0}\not\in \sigma_{1}(T).$ Then there exists $\epsilon>0$ such that $T - \lambda I$  is generalized
upper semi-Fredholm. For every $\lambda$, there exists  $\epsilon^{'}$ such that $T - \lambda^{'} I \in SF_{+}^{-}(\mathcal{H})$ and  $N(T- \lambda^{'} I) \subseteq \bigcap_{n=1}^{\infty}R[(T- \lambda^{'} I)^{n}]$   if  $0<|\lambda^{'}- \lambda|<\epsilon^{'}.$  Since $T\in \overline{CH(\mathcal{H})},$  it induces that ind$(T- \lambda I)\geq 0$  by Lemma \ref{B2}(3). Then $T- \lambda^{'} I$  is Weyl if $0<|\lambda^{'}- \lambda|<\epsilon.$ Since the Weyl's
theorem holds for $T$, then $T - \lambda^{'} I$  is Browder and hence $T - \lambda^{'} I$ is invertible if $0<|\lambda^{'}- \lambda|<\epsilon.$ It implies $\lambda \in$ iso$\sigma(T)\cup\rho(T),$ where $\rho(T)= \mathbb{C}\backslash \sigma(T).$ We claim that $\lambda \not\in$ iso$\sigma(T).$  If not, since
$T$ is finite isoloid and theWeyl's theorem holds for $T,$  it follows that $\lambda \in \pi_{00}= \sigma(T)\backslash \sigma_{w}(T).$ Then  $T - \lambda I$ is Browder. It is in contradiction to the fact that $T\in \overline{CH(\mathcal{H})}$ by Lemma \ref{B2}(2). Thus $\lambda \not\in\sigma(T).$ It induces that $\lambda_{0} \in$ iso$\sigma(T)\cup\rho(T).$ Using the same way, we prove that $T- \lambda_{0} I$ is invertible, which means that $\lambda \not\in\sigma(T).$\\

Conversely, suppose that $\sigma(T) = \sigma_{1}(T)$ and $\sigma(T)\cup  \partial D$ is connected. Since  $\sigma_{w}(T)= \sigma(T),$ it follows that $\sigma_{w}(T)\cup  \partial D$ is connected. Using the fact that  iso$\sigma(T)\cap \sigma_{1}(T)= \emptyset$ and $\sigma(T)=\sigma_{1}(T),$ we know that iso$\sigma(T)= \emptyset.$ Thus  $\sigma(T)\backslash\sigma_{b}(T)= \emptyset.$ If there exists $\lambda \in \rho_{SF}(T)$ such that ind$(T-\lambda I)<0,$ then $\lambda \not\in\sigma_{1}(T)$ hence  $\lambda \not\in\sigma(T),$  which means that $T-\lambda I$  is invertible.
It is in contradiction to the fact that ind$(T-\lambda I)<0.$ Hence for any $\lambda \in \rho_{SF}(T),$ ind$(T-\lambda I)\geq0.$ Using Lemma \ref{B2}, $T\in \overline{CH(\mathcal{H})}.$
\end{proof}
\begin{corollary} Suppose $T\in \overline{CH(\mathcal{H})}$ and  the a-Weyl's theorem holds for $T$. Then  a-Weyl's theorem holds for $f(T)$ for any $f\in H(T).$
\end{corollary}
\begin{proof} Since $T\in \overline{CH(\mathcal{H})},$ it induces that for each pair $\lambda, \mu \in \mathbb{C}\backslash\sigma_{SF_{+}}(T),$ ind$(T-\lambda I)$ind$(T-\mu I)\geq0.$ Theorem 2.2 in \cite{hd} tells us that the a-Weyl's theorem holds for $f(T)$ for any $f\in H(T).$
\end{proof}

\begin{theorem}\label{B1} If  $T\in CH(\mathcal{H}),$  then  $T$ and $T^{\ast}$ satisfy a-Browder's theorem.
\end{theorem}
\begin{proof} Since $T\in CH(\mathcal{H}),$ then $\sigma_{p}(T^{\ast})= \emptyset,$ it follows that $T^{\ast}$ has SVEP. Recall from \cite[Lemma 2.18]{d} that a
(necessary and) sufficient condition for an operator $T$  to satisfy a-Browder's theorem is that $T$ has SVEP at points $\lambda\not\in \sigma_{ea}(T);$ hence $T^{\ast} $ satisfies a-Browder's theorem. The following
argument shows $T$  also satisfies a-Browder's theorem. Evidently, $\sigma_{ea}(T)\subseteq \sigma_{ab}(T).$ Thus to
prove that $T$ satisfies a-Browder's theorem it would suffice to prove that $\sigma_{ab}(T)\subseteq \sigma_{ea}(T).$ Let  $\lambda\not\in \sigma_{ea}(T).$ Then $T-\lambda I$  is upper semi-Fredholm  and  ind$(T- \lambda I)\leq0.$ Since $T^{\ast}$ has SVEP, dsc$(T-\lambda I)<\infty$  \cite[Theorem 3.17 ]{a} $\Rightarrow$ ind$(T- \lambda I)\geq0.$ Thus ind$(T- \lambda I)=0$ and  $T- \lambda I$  is Fredholm. But then,
since  dsc$(T-\lambda I)<\infty,$  asc$(T-\lambda I)=$ dsc$(T-\lambda I)<\infty$
\cite[Theorem 3.4 ]{a}, which implies that $\lambda\not\in \sigma_{ab}(T).$
\end{proof}

The following example gives us an operator which satisfies a-Browder's theorem but
not Ces$\grave{a}$ro-hypercyclic.
\begin{example}  Let $T$  be defined by
\begin{center}
$T(\frac{x_{0}}{2}, \frac{x_{1}}{3}, \frac{x_{2}}{4}, ...)$  for all $(x_{n})\in l^{2}(\mathbb{N}).$
\end{center}
Then  $T$ is quasi-nilpotent, so has SVEP and consequently satisfies a-Browder's theorem. On the other hand, by  Proposition \ref{A2} the operator $T$  is not Ces$\grave{a}$ro-hypercyclic.
\end{example}

\begin{theorem} If $T\in CH(\mathcal{H}),$ then $T^{\ast}$ satisfies Weyl's theorem. If also
$\pi_{00}(T) \subseteq \pi_{00}(T^{\ast}),$ then $T$ satisfies a-Weyl's theorem.
\end{theorem}
\begin{proof} Evidently, if $T\in CH(\mathcal{H}),$ then  $p_{00}(T)= p_{00}(T^{\ast})=\pi_{00}(T^{\ast})= \emptyset.$ Since $T^{\ast}$  satisfies Browder's theorem, it follows that $T^{\ast}$ satisfies Weyl's theorem.\\
Since $p_{00}(T)\subseteq \pi_{00}(T)$ for every operator $T,$  and since operators $T\in CH,$ satisfy Browder's theorem, we have that $\sigma(T)\backslash \sigma_{w}(T)= p_{00}(T)\subseteq \pi_{00}(T).$ Hence, if  $\pi_{00}(T)\subseteq\pi_{00}(T^{\ast}),$ then  $\sigma(T)\backslash \sigma_{w}(T)= p_{00}(T)\subseteq \pi_{00}(T)\subseteq\pi_{00}(T^{\ast})= p_{00}(T^{\ast})= p_{00}(T),$ i.e., $T$ satisfies Weyl's
theorem. To complete the proof, we prove now that $T$ satisfies a-Weyl's theorem.\\
Observe that if $T^{\ast}$ has SVEP, then $\sigma(T)= \sigma_{a}(T)$ and $\pi_{00}(T)= \pi_{00}^{a}(T).$ Let $\lambda \not\in \sigma_{ea}(T).$  then $T- \lambda I$  is upper semi-Fredholm and ind$(T- \lambda I)\leq0.$ Arguing as in the proof of Theorem \ref{B1}, it is  seen that $T- \lambda I$  is Fredholm and ind$(T- \lambda I)=0,$ i.e., $\lambda \not\in \sigma_{w}(T).$  Since $\sigma_{w}(T)\supseteq \sigma_{ea}(T)$ for every
operator $T,$  we conclude that  $\sigma_{w}(T)= \sigma_{ea}(T).$ But then, since $T$ satisfies Weyl's theorem, $\sigma_{a}(T)\backslash \sigma_{ea}(T)= \sigma(T)\backslash \sigma_{w}(T)= \pi_{00}(T)= \pi_{00}^{a}(T).$
\end{proof}
\begin{corollary}  $T\in CH(\mathcal{H})$ satisfies a-Weyl's theorem if and only if $\pi_{00}(T)=\emptyset.$
\end{corollary}


\end{document}